\documentclass[a4paper,11 pt]{amsart}

\RequirePackage{amsmath, amssymb, amsthm, amsfonts}
\usepackage{fullpage}
\usepackage[english]{babel}
\usepackage{amsfonts}
\usepackage{latexsym}
\usepackage{amsthm}
\usepackage{amsopn}
\usepackage{amsmath}
\usepackage[all]{xy}
\usepackage{verbatim}
\usepackage{amssymb}
\usepackage{float}
\usepackage{eucal}
\usepackage[active]{srcltx}
\usepackage{hyperref}
\usepackage{manyfoot}
\usepackage{amscd}
\usepackage[dvips]{graphics}
\usepackage[dvips]{graphicx}
\usepackage{color}
\usepackage{cite}
\RequirePackage[all]{xy}
 \usepackage[normalem]{ulem}
\usepackage[a4paper]{anysize}\marginsize{3.5cm}{3.5cm}{1.3cm}{2cm}
\pdfpagewidth=\paperwidth \pdfpageheight=\paperheight
\usepackage{enumerate}
\usepackage[latin1]{inputenc}
\usepackage[all]{xy}
\usepackage{indentfirst}
\usepackage{fancyhdr}
\usepackage{textcomp}
\usepackage{graphicx}
\usepackage{enumerate}

\usepackage{tikz}

\usetikzlibrary{arrows}
\newtheorem{theo}{Theorem}[section]
\newtheorem{lem}[theo]{Lemma}
\newtheorem{prop}[theo]{Proposition}
\newtheorem{cor}[theo]{Corollary}
\newtheorem{defin}[theo]{Definition}

\theoremstyle{definition}

\newtheorem{rem}[theo]{Remark}

\newcommand{\PP}{\mathbb{P}}

\newcommand{\Aut}{{\mathrm{Aut}}}

\newcommand{\IP}{\mathbb{P}}

\newcommand{\Fix}{\operatorname{Fix}}

\newcommand{\Z}{\mathbb{Z}}
\newcommand{\C}{\mathbb{C}}
\newcommand{\ra}{\rightarrow}
\newcommand{\rank}{\operatorname{rk}}

\newcommand{\IF}{\mathbb{F}}

\newcommand{\cD}{\mathcal{D}}

\newcommand{\cG}{\mathcal{G}}
\newcommand{\cH}{\mathcal{H}}

\newcommand{\cL}{\mathcal{L}}
\newcommand{\cM}{\mathcal{M}}

\newcommand{\cO}{\mathcal{O}}

\newcommand{\id}{\operatorname{id}}

\newcommand{\Pic}{\operatorname{Pic}}

\title{Calabi-Yau 4-folds of Borcea--Voisin type from F-Theory }
\author{Andrea Cattaneo, Alice Garbagnati and Matteo Penegini}

\begin{document}


\begin{abstract}
In this paper, we apply Borcea--Voisin's construction and give new examples of Calabi--Yau fourfolds $Y$, which admit an elliptic fibration onto a smooth threefold $V$,  whose singular fibers of type $I_5$ lie above a del Pezzo surface $dP \subset V$. These are relevant models for F-theory according to \cite{BHV08I, BHV08II}. Moreover, at the end of the paper we will give the explicit equations of some of these Calabi--Yau fourfolds and their fibrations.
\end{abstract}

\Footnotetext{{}}{\textit{2010 Mathematics Subject
Classification}: 14J32, 14J35, 14J50}

\Footnotetext{{}} {\textit{Keywords}: {Calabi--Yau Manifolds, Elliptic Fibrations, Generalized Borcea--Voisin's construction, del Pezzo Surfaces, K3 surfaces, F-Theory}}


\maketitle


\section{Introduction}

New models of Grand Unified Theory (GUT) have recently been developed using F-theory, a branch of string theory which provides a geometric realization of strongly coupled Type IIB string theory backgrounds see e.g., \cite{BHV08I, BHV08II}.  In particular, one can compactify F-theory on an elliptically fibered manifold, i.e. a fiber bundle whose general fiber is a torus.

We are interested in some of the mathematical questions posed by F-theory - above all - the construction of some of these models. For us, F-theory  will be of the form $\mathbb{R}^{3,1} \times Y$, where $Y$ is a Calabi-Yau fourfold admitting an elliptic fibration with a section on a complex threefold $V$, namely:
  \[
\begin{xy}
\xymatrix{
E  \ar@^{(->}[r] & Y \ar^{\mathcal{E}}[d] \\
& V.
 }
\end{xy}
\]

In general, the elliptic fibers $E$ of $\mathcal{E}$ degenerate over a locus contained in a complex codimension one sublocus  $\Delta(\mathcal{E})$ of $V$, the discriminant of $\mathcal{E}$. Due to theoretical speculation in physics,  $\Delta(\mathcal{E})$ should contain del Pezzo surfaces above which the general fiber is a singular fiber of type $I_5$ (\verb|Figure 1|): see, for instance, \cite{BHV08I, BP17}.

The aim of this work is to investigate explicit examples of elliptically fibered Calabi--Yau fourfolds $Y$ with this property by using a generalized Borcea--Voisin construction. The original Borcea--Voisin construction is described independently in \cite{Bo97} and \cite{V93} , there the authors produce Calabi--Yau threefolds starting form a K3 surface and an elliptic curve. Afterwards generalization to higher dimensions are considered, see e.g. \cite{CH07}, \cite{Dil12}.
There are two ways to construct fourfolds of Borcea--Voisin type, by using involutions, either starting from a pair of K3 surfaces, or considering a Calabi--Yau threefolds and an elliptic curve. In this paper we will consider the former one. A first attempt to construct explicit examples of such  Calabi--Yau fourfolds $Y$ was done in \cite{BP17},  also using a generalized Borcea--Voisin's construction but  applied to a product of a Calabi--Yau threefold and an elliptic curve. In that case the Calabi--Yau threefold was a complete intersection $(3,3)$ in $\PP^5$ containing a del Pezzo surface of degree $6$, this construction was inspired by \cite{K}.

In order to construct a Calabi--Yau fourfold $Y$ with the elliptic fibration $\mathcal{E}$ as required one needs both a map to a smooth threefold $V$ whose generic fibers are genus 1 curves and a distinguished del Pezzo surface $dP$ in $V$. A natural way to produce these data is to consider two K3 surfaces $S_1$ and $S_2$ such that $S_1$ is the double cover of $dP$ and $S_2$ admits an elliptic fibration $\pi:S_2\ra\mathbb{P}^1$.  In this way we will obtain $\mathcal{E}:Y\ra V\simeq dP\times \mathbb{P}^1$.
To get $Y$ from $S_1$ and $S_2$ we need a non-symplectic involution on each surface. Since $S_1$ is a double cover of $dP$, it clearly admits the cover involution, denoted by $\iota_1$, while the involution $\iota_2$ on $S_2$ is induced by the elliptic involution on each smooth fiber of $\pi$.
Thus, $\left(S_1\times S_2\right)/\left(\iota_1\times \iota_2\right)$ is a singular Calabi--Yau fourfolds which admits a crepant resolution $Y$ obtained blowing up the singular locus.
It follows at once that there is a map $Y\ra (S_1/\iota_1)\times \mathbb{P}^1\simeq dP \times\mathbb{P}^1$ whose generic fiber is a smooth genus 1 curve and the singular fibers lies either on $dP\times \Delta(\pi)$ or on $C\times \mathbb{P}^1$ (where $C\subset dP$ is the branch curve of $S_1\ra dP$ and $\Delta(\pi)$ is the discriminant of $\pi$). The discriminant $\Delta(\pi)$ consists of a finite number of points and generically the fibers of $\mathcal{E}$ over $dP\times \Delta(\pi)$ are of the same type as the fiber of $\pi$ over $\Delta(\pi)$. Therefore the requirements on the singular fibers of $\mathcal{E}$ needed in F-Theory reduce to a requirements on the elliptic fibration $\pi:S_2\ra\mathbb{P}^1$.

Moreover we show that the choice of $S_1$ as double cover of a del Pezzo surface and of $S_2$ as elliptic fibration with specific reducible fibers can be easily modified to obtain Calabi--Yau fourfolds with elliptic fibrations with different basis (isomorphic to $S_1/\iota_1\times \mathbb{P}^1$) and reducible fibers (over $S_1/\iota_1\times \Delta(\pi)$).

Our first result (see Propositions \ref{prop: BV ef with singular fibers} and \ref{prop: Hodge numbers Y}) is
\begin{theo} Let $dP$ be a del Pezzo surface of degree $9-n$ and $S_1 \rightarrow dP$ a double cover with $S_1$ a K3 surface. Let $S_2\ra \mathbb{P}^1$ be an elliptic fibration on a K3 surface with singular fibers $mI_5+(24-5m)I_1$. The blow up $Y$ of $(S_1\times S_2)/(\iota_1\times \iota_2)$ along its singular locus is a crepant resolution. It is a Calabi--Yau fourfold which admits an elliptic fibration $\mathcal{E}:Y\ra dP\times \mathbb{P}^1$ whose deiscriminant contains $m$ copies of $dP$ above which the fibers are of type $I_5$. The Hodge numbers of $Y$ depends only on $n$ and $m$ and are
\begin{eqnarray*}\begin{array}{c}h^{1,1}(Y)=5+n+2m, \
h^{2,1}(Y)=2(15-n-m),\\
h^{2,2}(Y)=4(138-9n-19m+2nm),\
h^{3,1}(Y)=137-11n-22m+2nm.
\end{array}\end{eqnarray*}
\end{theo}

We also give more specific results on $Y$. Indeed, recalling that a del Pezzo surface is a blow up of $\mathbb{P}^2$
in $n$ points $\beta: dP\ra\mathbb{P}^2$, for $0\leq n\leq 8$, we
give a Weierstrass equation for the elliptic fibration $Y\ra
\beta(dP)\times \mathbb{P}^1$ induced by $\mathcal{E}$, see
\eqref{eq: Y weierstrass cF}. Moreover, in case $n=5,6$ we provide
the explicit Weierstrass equation of the fibration
$\mathcal{E}:Y\ra dP\times\mathbb{P}^1$, see \eqref{eq: Y if n=5}
and \eqref{eq: Y if n=6}.

In case $m=4$, there are two different choices for
$\pi:S_2\ra\mathbb{P}^1$. One of them is characterized by the
presence of a 5-torsion section for $\pi:S_2\ra \mathbb{P}^1$ and
in this case the K3 surface $S_2$ is a $2:1$ cover of the rational
surface with a level 5 structure, see \cite{BDGMSV17}. We observe that if $\pi:S_2\ra
\mathbb{P}^1$ admits a 5-torsion section, the same is true for
$\mathcal{E}$.\\

The particular construction of $Y$ enables us to find other two
distinguished fibrations (besides $\mathcal{E}$): one whose fibers are K3 surfaces and the
other whose fibers are Calabi--Yau threefolds of Borcea--Voisin
type. So $Y$ admits fibrations in Calabi--Yau manifolds of
any possible dimension.

The geometric description of these fibrations and their projective
realization is based on a detailed study of the linear systems of
divisors on $Y$. In particular we consider divisors $D_Y$ induced
by divisors on $S_1$ and $S_2$. We relate the dimension of the
spaces of sections of $D_Y$ with the one of the associated
divisors on  $S_1$ and $S_2$. Thanks to this study we are also
able to describe $Y$ as double cover of
$\mathbb{P}^2\times\mathbb{F}_4$ (where $\mathbb{F}_4$ is the
Hirzebruch surface $S_2/\iota_2$) and as embedded variety in
$\mathbb{P}^{59-n}$. The main results in this context are
summarized in Propositions \ref{prop: maps induced by D_Y} and
\ref{prop: maps induced by delta_Y}.

The paper is organized as follows. In Section \ref{sec:
preliminaries}, we recall the definition of Calabi--Yau manifold,
K3 surface and del Pezzo surface. Moreover, we describe
non-symplectic involutions on K3 surfaces. Finally in \ref{sec_BC}
we introduce the Borcea--Voisin construction.  Section \ref{sec:
The construction} is devoted to present  models $Y$ for the
F-theory described in the introduction. The Hodge number of $Y$
are calculated in Section \ref{sec: hodge numbers}. Section
\ref{sec: linear system} is devoted to the study of the linear
systems on $Y$. The results are applied in Section
\ref{sec_FibEmod} where several fibrations and projective models
of $Y$ are described. Finally, in Section \ref{sec: explicit} we
provide the explicit equations for some of these models and
fibrations.

\bigskip
\textbf{Acknowledgements.}
The authors would like to thank Sergio Cacciatori and Gilberto Bini for suggesting this problem at the meeting {\it Workshop ``Interazioni fra Geometria algebrica e Fisica teorica'' Villa Grumello, Como, January 2016} , Lidia Stoppino and Matteo Bonfanti for useful conversations and suggestions.
The second author is partially supported by FIRB 2012 ``\emph{Moduli spaces and their applications}''; the third author is partially supported by Progetto MIUR di Rilevante Interesse Nazionale \emph{Geometria
delle Variet$\grave{a}$ Algebriche e loro Spazi di Moduli} PRIN 2015. The authors were also
   partially supported by GNSAGA of INdAM.

\bigskip
\textbf{Notation and conventions.}
We work over the field of complex numbers $\mathbb{C}$.

\section{Preliminaries}\label{sec: preliminaries}

\begin{defin} A \emph{Calabi--Yau} manifold $X$ is a compact k\"ahler manifold with trivial canonical bundle such that $h^{i,0}(X)=0$ if $0< i < \dim X$.

 A  \emph{K3} surface  $S$ is a Calabi-Yau manifold of dimension $2$. The Hodge numbers of $S$ are uniquely determined by these properties and are $h^{0,0}(S)=h^{2,0}(S)=1$, $h^{1,0}(S)=0$, $h^{1,1}(S)=20$.
\end{defin}

\subsection{}
An involution $\iota$ on a K3 surface $S$ can be either
symplectic, i.e. it preserves the symplectic structure of the
surface, or not in this case we speak of non-symplectic involution.  In addition, an involution on a K3 surface is symplectic if and only if its fixed locus consists of isolated points; an involution
on a K3 surface is non-symplectic if and only if there are no
isolated fixed points on $S$. These remarkable
results depend on the possibility to linearize $\iota$ near the
fixed locus. Moreover, the fixed locus of an
involution on $S$ is smooth. In particular, the fixed locus of a
non-symplectic involution on a K3 surface is either empty or
consists of the disjoin union of curves.

From now on we consider only non-symplectic involutions $\iota$ on  K3 surfaces $S$.
As a consequence of the Hodge index theorem and of the adjunction
formula, if the fixed locus contains at least one curve $C$ of
genus $g(C):=g\geq 2$, then all the other curves in the fixed
locus are rational. On the other hand, if there is one curve of
genus 1 in the fixed locus, than the other fixed curves are either
rational curves or exactly one other curve of genus 1.

So one obtains that the fixed locus of $\iota$ on $S$ can be one
of the following:
\begin{itemize}
\item empty; \item the disjoint union of two smooth genus 1 curves
$E_1$ and $E_2$; \item the disjoint union of $k$ curves, such that
$k-1$ are surely rational, the other has genus $g\geq
0$.\end{itemize}

If we exclude the first two  cases
($Fix_\iota(S)=\emptyset$, $Fix_\iota(S)=E_1\coprod E_2$) the fixed
locus can be topologically described by the two integers $(g,k)$.

There is another point of view in the description of the
involution $\iota$ on $S$. Indeed $\iota^*$ acts on the second
cohomology group of $S$ and its action is related to the moduli
space of K3 surfaces admitting a prescribed involution; this is due to the
construction of the moduli space of the lattice polarized K3
surfaces. So we are interested in the description of the lattice
$H^2(S,\Z)^{\iota^*}$. This coincides with the invariant part of the N\'eron--Severi group $NS(S)^{\iota_*}$
since the automorphism is non-symplectic, and thus acts on $H^{2,0}(S)$ as $-\id_{H^{2,0}(S)}$, see \cite{nik}. The lattice
$H^2(S,\Z)^{\iota^*}$ of rank $r:=\rank(H^2(S,\Z)^{\iota^*})$ is
known to be 2-elementary, i.e. its discriminant group is
$(\Z/2\Z)^a$. Hence one can attach to this lattice the two
integers $(r,a)$. A very deep and important result on the
non-symplectic involutions on K3 surfaces is that each admissible pair of integers $(g,k)$ is uniquely associated to a pair of
integers $(r,a)$, see e.g.\ \cite{nik}.

We observe that for several admissible choices of $(r,a)$ this
pair uniquely determines the lattice $H^2(S,\Z)^{\iota^*}$, but
there are some exception.

The relation between $(g,k)$ and $(r,a)$ are explicitly given by
\begin{equation}\label{eq: relations (g,k), (r,a)} g=\frac{22-r-a}{2},\ \ k=\frac{r-a}{2}+1,\ \
r=10+k-g,\ \ a=12-k-g.
\end{equation}

\subsection{}
A surface $dP$ is called a \emph{del Pezzo surface} of degree $d$ if the anti-canonical bundle $-K_{dP}$ is ample and $K^2_{dP}=d$. Moreover we say that $dP$ is a \emph{weak del Pezzo} surface if $-K_{dP}$ is big and nef.

The anti-canonical map embeds $dP$ in $\PP^d$ as a surface of degree $d$. Another way to see $dP$ is as a blow up of $\PP^2$ in $9-d$ points in general position
\begin{equation}\label{eq_delPezzoP2}
\beta\colon dP \cong Bl_{9-d}(\PP^2) \longrightarrow \PP^2,
\end{equation}
see e.g., \cite{D13}.

\subsection{} A double cover of a del Pezzo surface $dP$ ramified along a smooth curve $C \in |-2K_{dP}|$ is a K3 surface $S$, endowed  with the covering involution $\iota$. Since $dP$ is not a symplectic manifold $\iota$ is non-symplectic. We can see $S$ as the minimal resolution of a double cover of $\PP^2$ branched along $\beta(C)$, which is a sextic with $9-d$ nodes.  Let us denote by $\rho'\colon S \rightarrow \PP^2$ the composition of the double cover with the minimal resolution. The ramification divisor of $\rho'$ is a genus $1+d$ smooth curve, which is the fixed locus of $\iota$.

\begin{defin}An \emph{elliptic fibration} $\mathcal{E}\colon Y \ra V$ is a surjective map with connected fibers between smooth manifolds such that: the general fiber of $\mathcal{E}$ is a smooth genus $1$ curve; there is a rational map $O \colon V \dashrightarrow Y$ such that $\mathcal{E} \circ O = id_{V} $. A \emph{flat elliptic fibration} is an elliptic fibration with a flat map $\mathcal{E}$. In particular a flat elliptic fibration has equidimensional fibers.
\end{defin}

\subsection{}\label{subsec: K3 with hyperelliptic invo} If $Y$ is a surface then any elliptic fibration is flat. Moreover, on $Y$ there is an involution $\iota$ which restricts to the elliptic involution on each smooth fiber. If $Y$ is a K3 surface, then $\iota$ is a non-symplectic involution.

\subsection{The Generalized Borcea--Voisin construction}\label{sec_BC} Let $X_i$, $i=1,2$ be a Calabi--Yau manifold endowed with an involution $\iota_i$ whose fixed locus has codimension $1$.
The quotient
\[ (X_1 \times X_2) / (\iota_1 \times \iota_2)
\]
admits a crepant resolution which is a Calabi--Yau manifold as well (see \cite{CH07}). We call \emph{Borcea--Voisin} of $X_1$ and $X_2$ the Calabi--Yau $BV(X_1,X_2)$ which is the blow up of $ (X_1 \times X_2) / (\iota_1 \times \iota_2)$ in its singular locus.

\subsection{} Let $b\colon \widetilde{X_1 \times X_2} \ra X_1 \times X_2$ be the blow up of  $ X_1 \times X_2$ in the fixed locus of $\iota_1 \times \iota_2$. Let $\tilde{\iota}$ be the induced involution  on $ \widetilde{X_1 \times X_2}$ and $q\colon  \widetilde{X_1 \times X_2} \rightarrow \widetilde{X_1 \times X_2}/{\tilde{\iota}}=:Y$ its quotient. The following commutative diagram:
\[\xymatrix{\widetilde{X_1 \times X_2} \ar[d]_{q} \ar[r]^{b} & X_1 \times X_2 \ar[d]\\
BV(X_1,X_2) \cong Y \ar[r] & (X_1 \times X_2) / (\iota_1 \times \iota_2),}\]
exhibit the Borcea--Voisin manifold as a smooth quotient.

\section{The construction}\label{sec: The construction}
\subsection{}

In the following we apply the just described Borcea--Voisin construction in order to get  a Calabi--Yau fourfold $Y$ together with a fibration  $\mathcal{E}\colon Y \longrightarrow V$ onto a smooth threefold $V$, with the following property: the general fiber of $\mathcal{E}$ is a smooth elliptic curve $E$, the discriminant locus of $\mathcal{E}$ contains a del Pezzo surface $dP$ and  for a generic point $p \in dP$ the singular fibers $\mathcal{E}^{-1}(p)$ is of type $I_5$ (see Figure 1).
\begin{center}
\begin{tikzpicture}[line cap=round,line join=round,>=triangle 45,x=0.35cm,y=0.35cm]
\clip(-5.5,0.1) rectangle (-0.1,5.);
\draw (1.82,0.96)-- (4.64,0.98);
\draw (3.94,0.74)-- (4.84,2.84);
\draw (4.98,2.16)-- (2.88,3.66);
\draw (4.02,3.64)-- (1.7,2.28);
\draw (1.88,3.38)-- (2.66,0.52);
\draw (-5.070087854421566,2.3602155818315578)-- (-1.9681467574189369,4.658927073789729);
\draw (-0.5843923132977713,3.8075232900966816)-- (-1.6748476802200527,0.34049514794709257);
\draw (-0.762046549892759,0.9667731457557286)-- (-4.512887329300956,0.9313878553839531);
\draw (-3.377801429369542,0.2774755617821425)-- (-4.480042967694254,3.5640041559443296);
\draw (-3.006518982335188,4.522691801137968)-- (-0.2689692460922668,2.5729354072620056);
\begin{scriptsize}
\draw[color=black] (3.2800344297313693,0.8496097668616738) node {$f$};
\draw[color=black] (4.684523322620731,1.843028739880978) node {$g$};
\draw[color=black] (4.136430096127321,3.281773459426177) node {$h$};
\draw[color=black] (2.7661970298937977,3.3502851127378532) node {$i$};
\draw[color=black] (2.0468246701211976,2.031435786488087) node {$j$};
\end{scriptsize}
\end{tikzpicture}
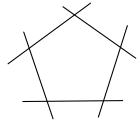
\centerline{\bf Figure 1: Fibre of type $I_5$}
\end{center}

\subsection{}\label{say_K3}
Let $S_1$ and $S_2$ be two $K3$ surfaces with the following properties
\begin{enumerate}
\item $S_1$ admits a $2 : 1$ covering $\rho': S_1 \longrightarrow
\IP^2$, branched along a curve $C$, which is a
(possibly singular and possibly reducible) sextic curve in
$\IP^2$.
\item $S_2$ admits an elliptic fibration $\pi: S_2 \longrightarrow \IP^1$,
with discriminant locus $\Delta(\pi)$.
\end{enumerate}

The surface $S_1$ has the covering involution $\iota_1$,
which is a non-symplectic involution. Moreover, if the branch curve $C\subset \IP^2$ is singular, then the
double cover of $\mathbb{P}^2$ branched along $C$ is singular. In this case the
K3 surface $S_1$ is the minimal resolution of this last singular surface.
The fixed locus of $\iota_1$ consists of the strict transform
$\tilde{C}$ of the branch curve, and possibly of some other smooth
rational curves, $W_i$ (which arise from the resolution of the triple points of $C$).  Moreover notice that if we choose $C$ to be a sextic with $n \leq 9$ nodes in general position then $\rho'$ factors through
\[ \rho\colon S_1 \stackrel{2:1}{\longrightarrow} dP := {\rm Bl}_n \IP^2,
\]
where $dP$ is a del Pezzo surface of degree $d = 9-n$.

The second $K3$ surface $S_2$ admits a non-symplectic involution too, as in \ref{subsec: K3 with hyperelliptic invo}.
This is the elliptic involution $\iota_2$, which acts on the
smooth fibers of $\pi$ as the elliptic involution of each
elliptic curve. In particular it fixes the 2-torsion group on each
fiber. Therefore, it fixes the zero section $O$, which is a rational curve,
and the trisection $T$ (not necessarily irreducible) passing through
the 2-torsion points of the fiber.

\subsection{}

Applying the Borcea--Voisin construction \ref{sec_BC} to $(S_1, \iota_1)$ and  $(S_2,  \iota_2)$ we obtain a smooth Calabi--Yau fourfold $Y$. In particular, the singular locus  of the quotient $X := (S_1 \times S_2) / (\iota_1 \times \iota_2)$ is the image of the fixed locus of the product involution  $\iota_1 \times \iota_2$.  As the involution acts componentwise we have
\[\Fix_{S_1 \times S_2} (\iota_1 \times \iota_2)  = \Fix_{S_1} \iota_1 \times \Fix_{S_2} \iota_2,\]
therefore the fix locus  consists of the \emph{disjoint union} of:
\begin{enumerate}
\item the surface $\tilde{C} \times O$, where $O \simeq \IP^1$ is the section of $\pi$;
\item the surface $\tilde{C} \times T$, where $T$ is the trisection of $\pi$; and eventually
\item the surfaces $\tilde{C} \times E_i$ (where $E_i \simeq \IP^1$ are the fixed components in the reducible fibers of $\pi$) and the surfaces $W_i \times O$, $W_i \times T$ and $W_i \times E_j$.
\end{enumerate}

As in \ref{sec_BC} we have the following commutative diagram.

\begin{equation}\label{de diagrame}\xymatrix{\widetilde{S_1 \times S_2} \ar[d]_q \ar[r]^b & S_1 \times S_2 \ar[d]\\
Y \ar[r] & X.}\end{equation}

\subsection{}\label{subsec: singular fibers} By construction the smooth fourfold $Y$ comes with several fibrations. Let us analyze one of them and we postpone the description of the other in Section \ref{sec_FibEmod}.

We have the fibration $ Y \rightarrow \IP^2 \times \IP^1$ induced by the covering $\rho'_d\colon S_1 \rightarrow \IP^2$ and the fibration $\pi\colon S_2 \rightarrow \IP^1$. Reacall from Paragraph \ref{say_K3} that we can specialize the fibration if we require that $\rho'$ is branched along a sextic with $n$ nodes in general position. This  further assuption yields
\[\xymatrix{ Y\ar[d]^{\varphi}\\ dP \times \IP^1,
\,}\]
where $dP$ is the del Pezzo surface obtained blowing up the nodes of the branch locus. The general fiber of $\varphi$ is an elliptic curve.  Indeed, let $(p, q) \in dP \times \IP^1$ with $p \notin C$ and $q \notin \Delta(\pi)$. Then $(\varphi)^{-1}(p, q)$ is isomorphic of the smooth elliptic curve $\pi^{-1}(q)$. Hence the singular fibers lies on points $(p,q)\in dP \times \IP^1$ of one of the following types: $p\in C$, $q\not \in \Delta(\pi)$; $p\not \in C$, $q \in \Delta(\pi)$; $p\in C$, $q\in \Delta(\pi)$. 
We discuss these three cases separately.

{\bf Case 1}  $(p , q) \in dP \times \IP^1$ with $p \notin C$ and $q \in \Delta(\pi)$. Clearly $\pi^{-1}(q)$ is  a singular curve, and since $p \notin C$, we get a singular fiber for $\varphi$
\begin{equation}\label{eq_singE}
\varphi^{-1}(p, q) \simeq \pi^{-1}(q),
\end{equation}

{\bf Case 2}  $(p , q) \in dP \times \IP^1$ with $p \in C$ and $q \notin \Delta(\pi)$. Consider first $(\rho \times \pi)^{-1}(p, q)$ in $S_1 \times S_2$. This is a single copy of $\pi^{-1}(q)$, which is a smooth elliptic curve, over the point $p \in C \subseteq S_1$. In addition,  this curve meets the fixed locus of $\iota_1 \times \iota_2$ in $4$ distinct points: one of them corresponds to the intersection with $C \times O$ and the other three correspond to the intersections with $C \times T$. Notice that  $\iota_1 \times \iota_2$ acts on $p \times \pi^{-1}(q)$ as the elliptic involution $\iota_2$, hence the quotient curve is a rational curve.  This discussion yields that $\varphi^{-1}(p,q)$ is a singular  fiber of type $I_0^*$, where the central rational components is isomorphic to the quotient of $\pi^{-1}(q)/\iota_2$ and the other four rational curves are obtained by blowing up the intersection points described above.

{\bf Case 3}  $(p , q) \in dP \times \IP^1$ with $p \in C$ and $q \in \Delta(\pi)$. This time, $(\rho \times \pi)^{-1}(p, q)$ is the singular fiber $\pi^{-1}(q)$. Moreover, the quotient of this curve by $\iota_2$ is determined by its singular fiber type.  If $\iota_2$ does not fix a component of $\pi^{-1}(q)$, then $(\rho \times \pi)^{-1}(p, q)$ meets the fixed locus of $\iota_1 \times \iota_2$: in a certain number of isolated points, depending on the fiber $\pi^{-1}(q)$ (which correspond to the intersection of the fiber with $O$ and $T$). On the other hand, if $\iota_2$ does fix a component  of $\pi^{-1}(q)$, then there are curves in  $(\rho \times \pi)^{-1}(p, q)$. In the later case $\phi^{-1}(p,q)$ contains a divisor.

In each of the previous case, the fiber over $(p,q)$ is not smooth and thus we obtain 
that the discriminant locus of $\varphi$ is
\[\Delta(\varphi) = (C \times \IP^1) \cup (dP \times \Delta(\pi)).\]

This discussion yields $\forall q \in \Delta(\pi)$ the surface $dP \times \{q\} \subset \Delta(\varphi)$ and for the generic point $p \in dP$ the fiber of $\varphi$ over $(p,q)$ are of the same type as the fiber of $\pi$ over $q$.
This implies the following Proposition.

\begin{prop}\label{prop: BV ef with singular fibers}  There exists a Calabi--Yau fourfold with an elliptic fibration over $dP \times \PP^1$ such that the discriminant locus contains a copy of $dP$.
If moreover we assume that the generic fiber above it is reduced, i.e. is of type $I_n$, $II$, $III$, $IV$, then it is possible to construct this elliptic fibration to be flat.
\end{prop}
\begin{proof} It remains to prove that for the fibers of type $I_n$, $II$, $III$, and $IV$ the fibration is flat. This follows by the analysis of case 3 since the involution $\iota_2$ does not fix any components of reduced fibers.
\end{proof}

\subsection{} We shall now discuss a special case of the elliptic fibration $\varphi$. Apparently, a good model for F-Theory (see Introduction and references there) is the one where the discriminant locus contains a del Pezzo surface over which there are $I_5$ singular fiber. Let us discuss this situation.

\begin{rem}\label{rem: no only I5} By Propostion \ref{prop: BV ef with singular fibers}  it is possbile to constrcut elliptic fibrations with fibers $I_5$. Nevertheless, it is not possible to obatin elliptic fibrations such that \emph{all} the singular fibers are of type $I_5$. Indeed there are two different obstructions:
\begin{enumerate}
\item  the fibers obtained in Case 2 of \ref{subsec: singular fibers} are of type $I_0^*$ and this does not depend on the choice of the properties of the elliptic fibration $S_2\ra\mathbb{P}^1$;
\item  the singular fibers as in Case 1 of \ref{subsec: singular fibers} depend only on the singular fibers of $S_2\ra\mathbb{P}^1$ and these can not be only of type $I_5$, indeed $24=\chi(S_2)$ is not divisible by 5.
\end{enumerate}
However, it is known that there exist elliptic K3 surfaces with $m$ fibers of type $I_5$ and all the other singular fibers of type $I_1$ for $m=1,2,3,4$, cf.\ \cite{Shim}. In this case the fibers of type $I_1$ are $24-5m$.
\end{rem}

\section{The Hodge numbers of \texorpdfstring{$Y$}{Y}}\label{sec: hodge numbers}

The aim of this Section is the computation of the Hodge numbers of the constructed fourfolds.

\subsection{}
By \eqref{de diagrame} the cohomology of $Y$ is given by the part of the cohomology of $\widetilde{S_1\times S_2}$ which is invariant under $(\iota_1\times\iota_2)^*$. The cohomology of
$\widetilde{S_1\times S_2}$ is essentially obtained as sum of two
different contributions: the pullback by $b^*$ of the chomology of
$S_1\times S_2$ and the part of the cohomology introduced by the
blow up of the fixed locus $Fix_{\iota_1\times \iota_2}(S_1\times
S_2)$. The fixed locus $Fix_{\iota_1\times \iota_2}(S_1\times
S_2)=Fix_{\iota_1}(S_1)\times Fix_{\iota_2}(S_2)$ consists of
surfaces, which are product of curves. So $b:\widetilde{S_1\times
S_2}\ra S_1\times S_2$ introduces exceptional divisors which are
$\mathbb{P}^1$-bundles over surfaces which are product of curves.
The Hodge diamonds of these exceptional 3-folds depends only on
the genus of the curves in $Fix_{\iota_1}(S_1)$ and
$Fix_{\iota_2}(S_2)$.

Since, up to an appropriate shift of the indices, the Hodge
diamond of $\widetilde{S_1\times S_2}$ is just the sum of the
Hodge diamond of $S_1\times S_2$ and of all the Hodge diamonds of
the exceptional divisors, the Hodge diamond of
$\widetilde{S_1\times S_2}$ depends only on the properties of the
fixed locus of $\iota_1$ on $S_1$ and of $\iota_2$ on $S_2$.
Denoted by $(g_i, k_i)$, $i=1,2$ the pair of integers which
describes the fixed locus of $\iota_i$ on $S_i$, we obtain that
the Hodge diamond of $\widetilde{S_1\times S_2}$ depends only on
the four integers $(g_1,k_1,g_2,k_2)$.

Now we consider the quotient 4-fold $Y$. Its cohomology is the
invariant cohomology of $\widetilde{S_1\times S_2}$ for the action
of $(\iota_1\times\iota_2)^*$. Since the automorphism induced by
$\iota_1\times \iota_2$ on $\widetilde{S_1\times S_2}$ acts
trivially on the exceptional divisors, one has only to compute the
invariant part of the cohomology of $S_1\times S_2$ for the action
of $(\iota_1\times \iota_2)^*$. But this depends of course only on
the properties of the action of $\iota_i^*$ on the cohomology of
$S_i$. We observe that $\iota_i^*$ acts trivially on
$H^0(S_i,\Z)$, and that $H^1(S,\Z)$ is empty. Denoted by
$(r_i,a_i)$, $i=1,2$ the invariants of the lattice
$H^2(S_i,\Z)^{\iota_i^*}$, these determine uniquely $H^*(S_1\times
S_2,\Z)^{(\iota_1\times\iota_2)^*}$.

Thus the Hodge diamond of $Y$ depends only on $(g_i,k_i)$ and
$(r_i,a_i)$, $i=1,2$. By \eqref{eq: relations (g,k), (r,a)}, it is
immediate that the Hodge diamond of $Y$ depends only either on
$(g_1,k_1,g_2,k_2)$ or on $(r_1,a_1,r_2,a_2)$.

This result is already known, and due to J. Dillies
who computed the Hodge numbers of the Borcea--Voisin of the product of two K3 surfaces by mean of the invariants $(r_1,a_1,r_2,a_2)$ in \cite{Dil12}:
\begin{prop}{\rm(\cite[Section 7.2.1]{Dil12})}\label{prop: dillies}
Let $\iota_i$ be a non-symplectic involution on $S_i$, $i=1,2$, such that its fixed locus is non empty and does not consists of two curves of genus 1. Let $Y$ be the Borcea--Voisin 4-fold of $S_1$ and $S_2$.
Then
\begin{eqnarray*}\begin{array}{l}
h^{1,1}(Y)=1+\frac{r_1r_2}{4}-\frac{r_1a_2}{4}-\frac{a_1r_2}{4}+\frac{a_1a_2}{4}+\frac{3r_1}2-\frac{a_1}{2}+\frac{3r_2}{2}-\frac{a_2}{2}\\
h^{2,1}(Y)=22-\frac{r_1r_2}{2}+\frac{a_1a_2}{2}+5r_1-6a_1+5r_2-6a_2\\
h^{2,2}(Y)=648+3r_1r_2+a_1a_2-30r_1-30r_2-12a_1-12a_2\\
h^{3,1}(Y)=161+\frac{r_1r_2}{4}+\frac{a_1a_2}{4}+\frac{r_1a_2}{4}+\frac{a_1r_2}{4}-\frac{13r_1}{2}-\frac{13r_2}{2}-\frac{11a_1}{2}-\frac{11a_2}{2}.\\
\end{array}\end{eqnarray*}
\end{prop}

\subsection{}
Now we apply these computations to our particular case: $S_1$ is the double cover of $\mathbb{P}^2$ branched along a sextic with $n$ nodes and $S_2$ is an elliptic K3 surface with $m$ fibers of type $I_5$. So we obtain the following proposition.

\begin{prop}\label{prop: Hodge numbers Y}
Let $m \geq 0$ be an integer, and suppose that $\pi\colon S_2 \rightarrow \IP^1$ in an elliptic fibration with singular fibers of type $mI_5 + (24 - 5m)I_1$. Then
\begin{eqnarray*}\begin{array}{l}
h^{1,1}(Y)=5+n+2m\\
h^{2,1}(Y)=2(15-n-m)\\
h^{2,2}(Y)=4(138-9n-19m+2nm)\\
h^{3,1}(Y)=137-11n-22m+2nm.\\
\end{array}
\end{eqnarray*}
\end{prop}

\proof In order to deduce the Hodge numbers of $Y$ by Proposition \ref{prop: dillies}, we have to compute the invariants $(g_i,k_i)$ of the action of $\iota_i$ on $S_i$ in our context. The surface $S_1$ is a $2:1$ cover of $\mathbb{P}^2$ branched on a sextic with $n$ nodes and $\iota_1$ is the cover involution, so the fixed locus of $\iota_1$ is isomorphic to the branch curve hence has genus $10-n$. So $(g_1,k_1)=(10-n,1)$ and thus $r_1=1+n$ and $a_1=1+n$.
The involution $\iota_2$ on $S_2$ is the elliptic involution, hence fixes the section of the fibration, which
is a rational curve, and the trisection passing through the 2
torsion points of the fibers. Moreover, $\iota_2$ does not fix
components of the reducible fibers. So $k_2=2$ and it remains to
compute the genus of the trisection. The Weierstrass equation of
the elliptic fibration $S_2$ is $y^2=x^3+A(t)x+B(t)$ and the
equation of the trisection $T$ is $x^3+A(t)x+B(t)=0$, which
exhibits $T$ as $3:1$ cover of $\mathbb{P}^1_{t}$ branched on the zero
points of the discriminant $\Delta(t)=4A(t)^3+27B(t)^2$. Under our
assumptions,  the discriminant has $m$ roots of multiplicity 5 and
$24-5m$ simple roots, so that $T$ is a $3:1$ cover branched in
$24-5m+m=24-4m$ points with multiplicity 2. Therefore, by
Riemann-Hurwitz formula, one obtains $2g(T)-2=-6+24-4m$, i.e.
$g(T)=10-2m$. Hence $k_2=2$, $g_2=10-2m$ and so $r_2=2+2m$ and
$a_2=2m$.\endproof


\section{Linear systems on \texorpdfstring{$Y$}{Y}}\label{sec: linear system}

\subsection{} Here we state some general results on linear systems on the product of varieties with trivial canonical bundle, which will be applied to $S_1\times S_2$.

Let $X_1$ and $X_2$ be two smooth varieties with trivial canonical bundle, and $\cL_{X_1}$ and $\cL_{X_2}$ be two line bundles on $X_1$ and $X_2$ respectively. Observe that we have a natural injective homomorphism
\[\begin{array}{ccc}
H^0(X_1, \cL_{X_1}) \otimes H^0(X_2, \cL_{X_2}) & \longrightarrow & H^0(X_1 \times X_2, \pi_1^* \cL_{X_1} \otimes \pi_2^* \cL_{X_2})\\
s \otimes t & \longmapsto & \pi_1^* s \cdot \pi_2^* t,
\end{array}\]
where the $\pi_i$'s are the two projections. We now want to determine some conditions which guarantee that this map is an isomorphism.

Using the Hirzebruch--Riemann--Roch theorem, we have that
\[\chi(X_1 \times X_2, \pi_1^* \cL_{X_1} \otimes \pi_2^* \cL_{X_2}) = \chi(X_1, \cL_{X_1}) \cdot \chi(X_2, \cL_{X_2}).\]

If $\cL_{X_1}$ and $\cL_{X_2}$ are nef and big line bundles such that $\pi_1^* \cL_{X_1} \otimes \pi_2^* \cL_{X_2}$ is still nef and big, then the above formula and Kawamata--Viehweg vanishing Theorem lead to
\[h^0(X_1 \times X_2, \pi_1^* \cL_{X_1} \otimes \pi_2^* \cL_{X_2}) = h^0(X_1, \cL_{X_1}) \cdot h^0(X_2, \cL_{X_2}).\]

However, we are interested also in divisors which are not big and nef, therefore we need the following result.

\begin{prop}\label{prop: iso section}
Let $X_1$, $X_2$ be two smooth varieties of dimension $n_1$ and $n_2$ respectively. Assume that they have trivial canonical bundle $\omega_{X_i} = \cO_{X_i}$ and that $h^{0,n_i-1}(X_i) = 0$. Let $D_i \subseteq X_i$ be a smooth irreducible codimension 1 subvariety. Then the canonical map
\[H^0(X_1, \cO_{X_1}(D_1)) \otimes H^0(X_2, \cO_{X_2}(D_2)) \stackrel{\psi}{\longrightarrow} H^0(X_1 \times X_2, \pi_1^* \cO_{X_1}(D_1) \otimes \pi_2^* \cO_{X_2}(D_2))\]
is an isomorphism.
\end{prop}
\proof
By K\"unnet formula \[\begin{array}{rl}
h^{0, n - 1}(X_1 \times X_2) = & h^{0, n_1 - 1}(X_1) \cdot h^{0, n_2}(X_2) + h^{0, n_1}(X_1) \cdot h^{0, n_2 - 1}(X_2) =\\
= & h^{0, n_1 - 1}(X_1) + h^{0, n_2 - 1}(X_2)=0,
\end{array}\]
where $n = n_1 + n_2 = \dim X_1 \times X_2$.

As already remarked the $\psi$ map is injective, so it suffices to show that the source and target spaces have the same dimension.

We begin with the computation of $h^0(X_i, \cO_{X_i}(D_i))$. From the exact sequence
\[0 \longrightarrow \cO_{X_i}(-D_i) \longrightarrow \cO_{X_i} \longrightarrow \cO_{D_i} \longrightarrow 0\]
we deduce the exact piece
\[\begin{array}{c}
H^{n_i - 1}(X_i, \cO_{X_i}) \longrightarrow\ H^{n_i - 1}(D_i, \cO_{D_i}) \longrightarrow H^{n_i}(X_i, \cO_{X_i}(-D_i)) \longrightarrow\\
\longrightarrow H^{n_i}(X_i, \cO_{X_i}) \longrightarrow 0.
\end{array}\]
Since $H^{n_i - 1}(X_i, \cO_{X_i}) = 0$ by hypothesis, we get by Serre duality that
\[h^0(X_i, \cO_{X_i}(D_i)) = h^{n_i}(X_i, \cO_{X_i}(-D_i)) = h^{n_i - 1}(D_i, \cO_{D_i}) + 1.\]

Now we pass to the computation of $h^0(X_1 \times X_2, \pi_1^* \cO_{X_1}(D_1) \otimes \pi_2^* \cO_{X_2}(D_2))$. Let $D = D_1 \times X_2 \cup X_1 \times D_2$; and observe that
\[\pi_1^* \cO_{X_1}(D_1) \otimes \pi_2^* \cO_{X_2}(D_2) = \cO_{X_1 \times X_2}(D).\]
By the previous part of the proof, we have that
\[h^0(X_1 \times X_2, \pi_1^* \cO_{X_1}(D_1) \otimes \pi_2^* \cO_{X_2}(D_2)) = h^{n - 1}(D, \cO_D) + 1,\]
so we need to compute $h^{n - 1}(D, \cO_D)$ in this situation. Consider the following diagram of inclusions
\[\xymatrix{X_1 \times D_2 \ar@{^(->}[r]^(0.6){i_1} & D\\
D_1 \times D_2 \ar@{^(->}[u] \ar@{^(->}[ur]^{i} \ar@{^(->}[r] & D_1 \times X_2, \ar@{^(->}[u]_{i_2}}\]
and the short exact sequence
\[0 \longrightarrow \cO_D \longrightarrow {i_1}_* \cO_{X_1 \times D_2} \oplus {i_2}_* \cO_{D_1 \times X_2} \longrightarrow i_* \cO_{D_1 \times D_2} \longrightarrow 0,\]
where
\[\begin{array}{ccc}
\cO_D & \longrightarrow & {i_1}_* \cO_{X_1 \times D_2} \oplus {i_2}_* \cO_{D_1 \times X_2}\\
 s & \longmapsto & (s_{|_{X_1 \times D_2}}, s_{|_{D_1 \times X_2}})
\end{array}\]
and
\[\begin{array}{ccc}
{i_1}_* \cO_{X_1 \times D_2} \oplus {i_2}_* \cO_{D_1 \times X_2} & \longrightarrow & i_* \cO_{D_1 \times D_2}\\
(s_1, s_2) & \longmapsto & {s_1}_{|_{D_1 \times D_2}} - {s_2}_{|_{D_1 \times D_2}}.
\end{array}\]
This sequence induces the exact piece
\[\begin{array}{c}
H^{n - 2}(D_1 \times D_2, \cO_{D_1 \times D_2}) \rightarrow H^{n - 1}(D, \cO_D) \rightarrow\\
\rightarrow H^{n - 1}(X_1 \times D_2, \cO_{X_1 \times D_2}) \oplus H^{n - 1}(D_1 \times X_2, \cO_{D_1 \times X_2}) \longrightarrow 0,
\end{array}\]
from which we have that
\[\begin{array}{rl}
h^{n - 1}(D, \cO_D)
\leq & h^{n - 1}(X_1 \times D_2, \cO_{X_1 \times D_2}) + h^{n - 1}(D_1 \times X_2, \cO_{D_1 \times X_2}) +\\
 & + h^{n - 2}(D_1 \times D_2, \cO_{D_1 \times D_2}).
\end{array}\]
These last numbers are easy to compute using K\"unneth formula:
\[\begin{array}{rl}
h^{n - 1}(X_1 \times D_2, \cO_{X_1 \times D_2})
= & \sum_{i = 0}^{n - 1} h^{0, i}(X_1) \cdot h^{0,n - 1 - i}(D_2) =\\
= & h^{0, n_1}(X_1) \cdot h^{0, n_2 - 1}(D_2) =\\
= & h^{0, n_2 - 1}(D_2);\\
h^{n - 1}(D_1 \times X_2, \cO_{D_1 \times X_2}) = & h^{0, n_1 - 1}(D_1);\\
h^{n - 2}(D_1 \times D_2, \cO_{D_1 \times D_2}) = & h^{0, n - 2}(D_1 \times D_2) =\\
= & \sum_{i = 0}^{n - 2} h^{0, i}(D_1) \cdot h^{0,n - 2 - i}(D_2) =\\
= & h^{0, n_1 - 1}(D_1) \cdot h^{0, n_2 - 1}(D_2).
\end{array}\]
where we used the trivial observation that $h^{0,k}(D_i)=0$ if $k\geq n_i$.

Finally, we have the following chain of inequalities:
\[\begin{array}{l}
(h^{n_1 - 1}(D_1, \cO_{D_1}) + 1) (h^{n_2 - 1}(D_2, \cO_{D_2}) + 1) =\\
= h^0(X_1, \cO_{X_1}(D_1)) \cdot h^0(X_2, \cO_{X_2}(D_2)) \leq\\
\leq h^0(X_1 \times X_2, \cO_{X_1 \times X_2}(D)) =\\
= h^{n - 1}(D, \cO_D) + 1 \leq\\
\leq h^{0, n_1 - 1}(D_1) + h^{0, n_2 - 1}(D_2) + h^{0, n_1 - 1}(D_1) \cdot h^{0, n_2 - 1}(D_2) + 1 =\\
= (h^{n_1 - 1}(D_1, \cO_{D_1}) + 1) (h^{n_2 - 1}(D_2, \cO_{D_2}) + 1),
\end{array}\]
from which the Proposition follows.
\endproof

\subsection{}
In particular, this result applies when $X_1$ and $X_2$ are $K3$ surfaces or, more generally, when they are Calabi--Yau or hyperk\"ahler manifolds.

By induction, it is easy to generalize this result to a finite number of factors. Notice that we require $D_i$ to be smooth in order to use K\"unneth formula. Indeed, there is a more general version of Proposition \ref{prop: iso section} for line bundles. Namely, if $\cL_i$ are  globally generated/base point free  line bundles over $X_i$ then their  linear systems $|\cL_i|$ have, by Bertini's theorem, a smooth irreducible member, and we can apply Proposition \ref{prop: iso section}.

Let us denote  $D_1+D_2:=\pi_1^*\mathcal{O}(D_1)+\pi_2^*\mathcal{O}(D_2)$. The linear system $|D_i|$ naturally defines the map $\varphi_{|D_i|}\colon X_i\ra\mathbb{P}^{n_i}$. Denoted by $\sigma_{n_1,n_2}\colon\mathbb{P}^{n_1}\times\mathbb{P}^{n_2}\ra\mathbb{P}^{n_1n_2+n_1+n_2}$ the Segre embedding, Proposition \ref{prop: iso section} implies that $\varphi_{|D_1+D_2|}$ coincides with $\sigma_{n_1,n_2}\circ \left(\varphi_{|D_1|}\times\varphi_{|D_2|}\right)$.

\begin{cor}\label{cor: D_i on S_i and h^0(S_1xS_2)} Let $S_i$, $i=1,2$ be two K3 surfaces and  $D_i$ be an irreducible smooth curve of genus $g_i$ on $S_i$. Then $h^0(S_1\times S_2, D_1+D_2)=(g_1+1)(g_2+1)$.
\end{cor}
\subsection{ }\label{subsection: linear systems on Y}

Use the same notation as in Section 3 diagram \eqref{de diagrame}. On $S_1 \times S_2$, let  $D$  be an invariant divisor (resp.\ an invariant line bundle $\cD$) with respect to the $\iota_1 \times \iota_2$ action. Moreover,  denote by $D_Y$ the divisor on $Y$ such that $q^* D_Y = b^* D$ (resp.\ $\cD_Y$ is the line bundle such that $q^* \cD_Y = b^* \cD$).

Since $q$ is a double cover branched along a codimension 1 subvariety $B$, it is uniquely defined by a line bundle $\mathcal{L}$ on $Y$ such that $\mathcal{L}^{\otimes 2}=\mathcal{O}_Y(B)$ and we have
\[
H^0(\widetilde{S_1 \times S_2}, q^* \cM)
=  H^0(Y, \cM) \oplus H^0(Y, \cM \otimes \cL^{\otimes -1}).
\]
 for any line bundle $\mathcal{M}$ on $Y$.

The isomorphism
$H^0(\widetilde{S_1 \times S_2}, b^* \cD) \simeq H^0(S_1 \times S_2, \cD)$
yields
\[\begin{array}{rl}
H^0(S_1 \times S_2, \cD) \simeq &  H^0(\widetilde{S_1 \times S_2}, q^* \cD_Y) \\
\simeq & H^0(Y, \cD_Y) \oplus H^0(Y, \cD_Y \otimes \cL^{\otimes -1}).
\end{array}\]

As a consequence, one sees that
the space $H^0(Y, \cD_Y)$ corresponds to the invariant subspace of $H^0(S_1 \times S_2, \cD)$ for the $\iota^*$ action, while $H^0(Y, \cD_Y \otimes \cL^{-1})$ corresponds to the anti-invariant one. This yields at once the following commutative diagram:
\begin{equation}\label{eq: diagram of maps}
\xymatrix{\widetilde{S_1 \times S_2} \ar[r]^b \ar[d]_q & S_1 \times S_2 \ar[d] \ar[r]^(0.35){\varphi_{|\cD|}} & \IP(H^0(S_1 \times S_2, \cD)^{\vee}) \ar[d]\\
Y \ar[r] \ar@/_1pc/[rr]_{\varphi_{|\cD_Y|}} & X \ar[r] & \IP(H^0(Y, \cD_Y)^{\vee}),}
\end{equation}
where the vertical arrow on the right is the projection on $\IP(H^0(Y, \cD_Y)^{\vee})$ with center $\IP(H^0(Y, \cD_Y \otimes \cL^{-1})^{\vee})$ (observe that both these two spaces are pointwise fixed for the induced action of $\iota$ on $\IP(H^0(S_1 \times S_2, \cD)^{\vee})$).

In what follows we denote by $D_Y$ and $L$ the divisors such that
$\mathcal{D}_Y=\mathcal{O}(D_Y)$ and $\mathcal{L}=\mathcal{O}(L)$, so $L$ is half of the branch divisor.

\subsection{} Let $D_i$ be a smooth irreducible curve on $S_i$
such that the divisor $D_i$ is invariant for $\iota_i$. Then
$\iota_i^*$ acts on $H^0(S_i,D_i)^{\vee}$. Let us denote by
$H^0(S_i,D_i)_{\pm 1}$ the eigenspace relative to the eigenvalue
$\pm 1$ for the action of $\iota_i$ on $H^0(S_i,D_i)$. Let $h_i$
be the dimension of $\mathbb{P}(H^0(S_i,D_i)_{+1}^{\vee})$. It holds
\begin{cor}\label{cor: dimensions of D1D2 on Y}
Let $S_i$, $D_i$, $D_Y$, $L$ and $h_i$ be as above. Then
$\varphi_{|D_Y|}:Y\ra\mathbb{P}^{N}$ where
$N:=(h_1+1)(h_2+1)+(g(D_1)-h_1)(g(D_2)-h_2)-1$ and
$\varphi_{|D_Y-L|}:Y\ra\mathbb{P}^{M}$ where
$M:=(h_1+1)(g(D_2)-h_2)+(g(D_1)-h_1)(h_2+1)-1$.
\end{cor}
\proof By Corollary \ref{cor: D_i on S_i and h^0(S_1xS_2)} the map
$\varphi_{|D_1+D_2|}$ is a map from $S_1\times S_2$ to the Segre
embedding of $\mathbb{P}(H^0(S_1,D_1)^{\vee})$ and
$\mathbb{P}(H^0(S_2,D_2)^{\vee})$. The action of the automorphism
$\iota_1\times\iota_2$ on $H^0(S_1\times S_2, D_1+D_2)$ is induced
by the action of $\iota_i$ on $H^0(S_i,D_i)$ and in particular
$H^0(S_1\times S_2, D_1+D_2)_{+1}=H^0(S_1,D_1)_{+1}\otimes
H^0(S_2,D_2)_{+1}\oplus H^0(S_1,D_1)_{-1}\otimes
H^0(S_2,D_2)_{-1}$, whose dimension is
$(h_1+1)(h_2+1)+(g(D_1)-h_1)(g(D_2)-h_2)$. By Section
\ref{subsection: linear systems on Y}, the divisors $D_Y$ and
$D_Y-L$ define on $Y$ two maps whose target space is the
projection of $\mathbb{P}(H^0(S_1\times S_2,D)^{\vee})$ to the
eigenspaces for the action of $\iota_1\times \iota_2$ and the image
is the projection of $\varphi_{|D|}(S_1\times S_2)$. So the target
space of $\varphi_{|D_Y|}$ is $\mathbb{P}(H^0(S_1\times S_2,
D_1+D_2)_{+1}^{\vee})$, whose dimension is
$(h_1+1)(h_2+1)+(g(D_1)-h_1)(g(D_2)-h_2)-1$. Similarly one
concludes for $\varphi_{|D_Y-L|}$.\endproof

\begin{lem}\label{lem: dimensions deltai on Y}
Let $D_i$ be an effective divisor on $S_i$ invariant for $\iota_i$ and $h_i$ be the dimension of $\mathbb{P}(H^0(S_i,D_i)_{+1}^{\vee})$ for $i=1,2$. Denote by $\delta_{D_i}$ the divisor on $Y$ such that $q^*(\delta_{D_i})=b^*(\pi_i^*(D_i))$. Then $$H^0(S_1\times S_2, \pi_i^*(D_i))\simeq H^0(S_i,D_i)\mbox{  and  }\dim (\mathbb{P}(H^0(Y,\delta_{D_i})))=h_i,$$
for $i=1,2$.
\end{lem}

\section{Projective models and fibrations}\label{sec_FibEmod}
The aim of this section is to apply the general results of the
previous sections to our specific situation. So, let
$(S_1,\iota_1)$ and $(S_2,\iota_2)$ be as in Section \ref{say_K3} (i.e. $S_1$
is a double cover of $\mathbb{P}^2$, $\iota_1$ is the cover
involution, $S_2$ is an elliptic fibration and $\iota_2$ is the
elliptic involution). We now consider some interesting divisors
on $S_1$ and $S_2$.

\subsection{} Let $h\in \Pic(S_1)$ be the pullback of the
hyperplane section of $\mathbb{P}^2$ by the generically $2:1$ map
$\rho'\colon S_1\ra\mathbb{P}^2$.
The divisor $h$ is a nef and big divisor on $S_1$ and the map
$\varphi_{|h|}$ is generically $2:1$ to the image (which is
$\mathbb{P}^2$). The action of $\iota_1$ is the identity on $H^0(S_1,h)^{\vee}$, since $\iota_1$ is the cover involution.

We recall that the branch locus of $\rho'$
is a sextic with $n$ simple nodes in general
position, for $0\leq n\leq 8$. As explained in Section 3, in order to construct a smooth
double cover we first blow up $\mathbb{P}^2$ at the $n$ nodes of
the sextic obtaining a del Pezzo surface $dP$. Thus on $S_1$ there are $n$ rational curves, lying over
these exceptional curves. We denote these curves by
$R_i$, $i=1,\ldots,n $. We will denote by $H$ the divisor
$3h-\sum_{i=1}^nR_i$ if $n\geq 1$ or the divisor $3h$ if $n=0$. Observe that $H$ is the strict transform of the nodal sextic in $\PP^2$.

For a generic choice of $S_1$ the Picard group of $S_1$ is generated
by $h$ and $R_i$. The divisor $H$ is an ample divisor, because it has a positive intersection with all the effective $-2$ classes. Moreover,  $H^2=18-2n>2$, if $n\leq 7$. By \cite{SD}, this divisor can not be
elliptic and so the map $\varphi_{|H|}$ is $1:1$ onto its image
in  $\mathbb{P}^{10-n}$.

The divisor $\frac{1}{2}\rho_*(H)$ is the anticanonical divisor of the del Pezzo surface $dP$, which embeds $dP$ in $\mathbb{P}^{9-n}=\mathbb{P}(H^0(dP,\frac{1}{2}\rho_*(H))^{\vee})$. Since $\iota_1$ is the cover involution of $\rho$, the action of $\iota_1^*$ on $H^0(S_1,H)^{\vee}$ has a $(10-n)$-dimensional eigenspace for the eigenvalue $+1$ and a $1$-dimensional eigenspace for the eigenvalue $-1$. Observe that with this description, the projection $\PP(H^0(S_1, H)^{\vee}) \rightarrow \PP(H^0(S_1, H)^{\vee}_{+1})$ from the point $\PP(H^0(S_1, H)^{\vee}_{-1})$ coincides with the double cover $\rho$.

Notably, if $n=6$, the del Pezzo surface $dP$ is a cubic surface in $\PP^3_{(x_0:x_1:x_2:x_3)}$, whose equation is
$f_3(x_0:x_1:x_2:x_3)=0$. In this case the divisor $H$ embeds the
K3 surface $S_1$ in $\mathbb{P}^4$ as complete intersection of a
quadric  with equation $x_4^2=g_2(x_0:x_1:x_2:x_3)$ and the cubic
$f_3(x_0:x_1:x_2:x_3)=0$ and $\iota_1$ acts multiplying $x_4$ by $-1$.

\subsection{} Let $S_2$ be a K3 surface with an elliptic fibration. Generically $\Pic(S_2)$ is spanned
by the divisors $F$ and $O$, the class of the fiber and the class
of the section respectively. If $S_2$ has some other properities,
for example some reducible fibers, then there are other divisors
on $S_2$ linearly independent from $F$ and $O$. In any case, it is
still true that $\langle F,O\rangle$ is primitively embedded in
$\Pic(S_2)$. We consider two divisors on $S_2$: $F$ and
$4F+2O$.

The divisor $F$ is by definition the class of the fiber of the
elliptic fibration on $S_2$, so that $\pi=\varphi_{|F|}:S_2\ra
\mathbb{P}^1$ is the elliptic fibration on $S_2$. In particular
$F$ is a nef divisor, but it is not big, and it is invariant for
$\iota_2$ (since $\iota_2$ preserves the fibration). Moreover $\iota_2$ preserves each fiber of the fibration, therefore $\iota_2^*$ acts as the identity on $H^0(S_2,F)^{\vee}$.

It is easy to see that the divisor $4F+2O$ is a nef and big divisor. The map $\varphi_{|4F+2O|}$ contracts the zero section and possibly the non
trivial components of the reducible fibers of the fibration. We see that
\[\varphi_{|4F+2O|}\colon S_2 \stackrel{2:1}{\longrightarrow} \varphi_{|4F+2O|}(S_2)
\]
is a double cover, where  $\varphi_{|4F+2O|}(S_2)$ is the cone over a rational normal
curve of degree 4 in $\mathbb{P}^5$. Blowing up of the vertex of $\varphi_{|4F+2O|}(S_2)$ we obtain a surface isomorphic to the Hirzebruch surface $\mathbb{F}_4$. The involution $\iota_2$ is the associated cover involution, this means that
$\iota_2^*$ acts as the identity on $H^0(S_2,4F+2O)^{\vee}$.

\subsection{}
We observe that the divisors $h$, $H$, $F$ and $4F+2O$ are
invariant for the action of $\iota_i$ for some $i$.
So by Corollary \ref{cor: dimensions of D1D2 on Y} we get the following
\begin{prop}\label{prop: maps induced by D_Y} Let $Y$ and the divisors on $Y$ be as above, then
\begin{enumerate}
\item
the map
\[
\xymatrix{\varphi_{|(h+F)_Y|}: &Y\ar[rr]\ar[rd]&& \mathbb{P}^5\\
&&\mathbb{P}^2\times\mathbb{P}^1\ar[ru]^{\sigma_{2,1}}}\]
is an elliptic fibration on the image of
$\mathbb{P}^2\times\mathbb{P}^1$ by the Segre embedding;

\item
the map
\[
\xymatrix{\varphi_{|(H+F)_Y|}: &Y\ar[rr]\ar[rd]&& \mathbb{P}^{19-2n}\\
&&\mathbb{P}^{9-n}\times\mathbb{P}^1\ar[ru]^{\sigma_{9-n,1}}}\]
is the same elliptic fibration as in (1) with different projective model of the basis, i.e. the image of $dP\times \mathbb{P}^1$ via $\sigma_{9-n,1}$;
\item the map
\[
\xymatrix{\varphi_{|(h+(4F+2O))_Y|}: &Y\ar[rr]\ar[rd]&& \mathbb{P}^{17}\\
&&\mathbb{P}^2\times\mathbb{P}^5\ar[ru]^{\sigma_{2,5}}}\]
is a generically $2:1$ map onto its image contained in $\sigma_{2,5}(\mathbb{P}^2\times \mathbb{P}^5)$;
\item the map
\[
\xymatrix{\varphi_{|(H+(4F+2O))_Y|}: &Y\ar[rr]\ar[rd]&& \mathbb{P}^{59-6n}\\
&&\mathbb{P}^{9-n}\times\mathbb{P}^5\ar[ru]^{\sigma_{9-n,5}}}\]
is  birational onto its image contained in $\sigma_{9-n,5}(\mathbb{P}^{9-n}\times \mathbb{P}^5)$.
\end{enumerate}
\end{prop}
\proof The points (1) and (2) are proved in Section \ref{subsec:  fibrations}. The points (3) and (4) are proved in Section \ref{subsec: projective models}.\endproof

\begin{prop}\label{prop: maps induced by delta_Y} Using the same notation as for Lemma \ref{lem: dimensions deltai on Y} we have:
\begin{enumerate}
\item
$\varphi_{|\delta_h|}:Y\ra\mathbb{P}^2$ is an isotrivial fibration in K3 surfaces whose generic fiber is isomorphic to $S_2$.
\item
$\varphi_{|\delta_H|}:Y\ra\mathbb{P}^{9-n}$ is the same fibration as in (1) with a different projective model of the basis.
\item
$\varphi_{|\delta_F|}:Y\ra\mathbb{P}^1$ is a fibration in Calabi--Yau 3-folds whose generic fiber is the Borcea--Voisin of the K3 surface $S_1$ and the elliptic fiber of the fibration $\pi$.
\item
$\varphi_{|\delta_{4F+2O}|}:Y\ra\mathbb{P}^5$ is an isotrivial fibration in K3 surfaces whose generic fiber is isomorphic to $S_1$.
\end{enumerate}
\end{prop}
\proof The proof is explained in Section \ref{subsec:  fibrations}, where all the previous maps are described in details.\endproof

\subsection{Fibrations on \texorpdfstring{$Y$}{Y}}\label{subsec: fibrations}

As the natural map $\rho' \times \pi: S_1 \times S_2 \longrightarrow \IP^2 \times \IP^1$ satisfies $(\rho \times \pi) \circ \iota = \rho \times \pi$, we have an induced map
$X \longrightarrow \IP^2 \times \IP^1$.
The composition of this map with the resolution $Y \longrightarrow X$ and with the two projections then gives the following:
\begin{enumerate}
\item an elliptic fibration $\mathcal{E}: Y \longrightarrow \IP^2 \times \IP^1$;
\item a $K3$-fibration $\cG: Y \longrightarrow \IP^2$;
\item a fibration in elliptically fiberd threefolds $\cH: Y \longrightarrow \IP^1$.
\end{enumerate}

We describe these fibrations:

\medskip

{\bf (1)} The map $\mathcal{E}: Y \longrightarrow \IP^2 \times \IP^1$ is induced by the divisor $(h+F)_Y$ since $\varphi_{|h|}:S+1ra\mathbb{P}^2$ and $\varphi_{|F|}:S_1\ra\mathbb{P}^1$. We already described the properties and the singular fibers for this fibration in \ref{subsec: singular fibers}.

\medskip

The composition of $\varphi_{|H|}(S_1)$ and the projection to the invariant subspace of $\mathbb{P}^{10-n}$ exhibits $S_1$ as double cover of the del Pezzo surface $dP$ anticanonically embedded in $\mathbb{P}^{9-n}$. The del Pezzo surface $dP$ is the blow up of $\mathbb{P}^2$ in $n$ points and the double cover $S_1\ra dP$ corresponds (after the blow up) to the double cover $\varphi_{|h|}:S_1\ra\mathbb{P}^2$ since $H=3h-\sum_{i=1}^nR_i$. Thus, the map $\varphi_{|(H+F)_Y|}$ is the same fibration as $\varphi_{|(h+F)_Y|}$, with a different model for the basis (which is now $dP\times \mathbb{P}^1$).

\medskip

{\bf (2)}
The map $\cG: Y \longrightarrow \IP^2$ is induced by $\delta_h$. The fiber of these fibrations are isomorphic to $S_2$ since we have the following commutative diagram

\[
\xymatrix{ S_1\times S_2\ar[r]\ar[d]^{/\iota_1\times \iota_2}&S_1\ar[d]^{\rho}\ar[dr]^{\varphi_{|h|}}\\
X\ar[r]&dP\ar[r]&\mathbb{P}^2.}
\]

The singular fibers of $\cG$ lie over the branch curve $C\subset\mathbb{P}^2$ of the double cover $S_1\ra\mathbb{P}^2$.
Let $P\in C$. It is easy to see that $(\rho' \times \pi)^{-1}(pr_{\IP^2}^{-1}(P))$ is given by $P \times S_2$, and so in the quotient $X$ we see a surface isomorphic to $S_2 / \iota_2$, which is a surface obtained from $\IF_4$ by mean of blow ups. Moreover, under the blow up $Y \longrightarrow X$ we add a certain number of ruled surfaces: these last are all disjoint one from each other, and meet the blow up of $\IF_4$ on the base curve of the rulings, i.e.\ on the section $O$, on the trisection $T$ and possibly on the rational fixed components $E_i$ (which are necessarily contained in reducible not-reduced fibers).\\

For the same reason as above, $\varphi_{|\delta_{H}|}$ is the fibration $\cG$ with a different description of the basis.

\medskip

{\bf (3)} The fibration $\cH$ is induced by  $\delta_F$. For every $t\in \mathbb{P}^1$ we denote by $F_t$ the elliptic fiber of $S_2\ra\mathbb{P}^1$ over $t$.  The inclusion $S_1\times F_t\subset S_1\times S_2$ induces
$$\xymatrix{& S_1\times F_t\ar@{^(->}[r]\ar[d]^{/\iota_1\times (\iota_2)_{|F_t}}&S_1\times S_2\ar[d]^{/\iota_1\times \iota_2}\\
BV(S_1,F_t)\ar[r]&(S_1\times F_t)/(\iota_1\times (\iota_2)_{|F_t})\ar@{^(->}[r]&X&Y\ar[l]}$$
So the fibers of $\varphi_{\delta_F}$ are Borcea--Voisin Calabi--Yau 3-folds which are elliptically fibered by definition.
The singular fibers lie on $\Delta(\pi)$.

\medskip

{\bf (4)} Moreover there is another K3-fibration. Indeed, the map $\varphi_{|\delta_{4F+2O}|}$ gives an isotrivial fibration in K3 surfaces isomorphic to $S_1$ and with basis the cone over the rational normal curve in $\mathbb{P}^4$, by the diagram
$$\xymatrix{ S_1\times S_2\ar[r]\ar[d]^{/\iota_1\times \iota_2}&S_2\ar[d]^{/\iota_2}\\
X\ar[r]&(S_2/\iota_2)\ar[r] & \mathbb{P}^5.}$$

\subsection{Projective models}\label{subsec: projective models}
By the diagram
$$\xymatrix{
& S_1\times S_2\ar[d]_{2:1}\ar[rrr]_{4:1}^{\varphi_{|h|}\times \varphi_{|4F+2O|}} & & & \mathbb{P}^2\times\mathbb{P}^5\ar@{^(->}[r]_{\sigma_{2,5}}&\mathbb{P}^{17} \\
Y \ar[r] & X \ar[rrru]_{2:1} &  &}
$$
we can describe the map induced by the linear system $|(h + 4F + 2O)_Y|$ on $Y$ as a double cover of the image (under the Segre embedding of the ambient spaces) of $\varphi_{|h|}(S_1) \times \varphi_{|4F + 2O|}(S_2)$, which is the product of $\PP^2$ with the cone over the rational normal curve of degree $4$. This map is generically $2:1$, and its branch locus is given by the union of the product of the sextic curve in $\PP^2$ with the vertex of the cone 
(the fiber over such points is a curve) and the product of the sextic with the trisection; the generic fiber is a single point, but there may be points where the fiber is a curve. The last case occurs only if the fibration $\pi:S_2\ra\mathbb{P}^1$ has reducible non-reduced fibers.

To describe the map induced by $|(H + 4F + 2O)_Y|$ we use the following diagram
$$\xymatrix{
& S_1\times S_2\ar[d]_{2:1}\ar[rrr]_{2:1}^{\varphi_{|H|}\times \varphi_{|4F+2O|}} & & & \mathbb{P}^{10-n}\times\mathbb{P}^5\ar@{^(->}[r]_{\sigma_{10-n,5}} \ar[d] &\mathbb{P}^{65-6n} \\
Y \ar[r] & X \ar[rrr]_{1:1}^{\varphi_{|(H + 4F + 2O)_Y|}} &  & & \mathbb{P}^{9-n} \times \mathbb{P}^5\ar@{^(->}[r]_{\sigma_{9-n,5}}&\mathbb{P}^{59-6n}}.
$$
where $\mathbb{P}^{10-n}\times\mathbb{P}^5\ra \mathbb{P}^{9-n}\times\mathbb{P}^5$ is induced by the projection of $\mathbb{P}^{10-n}=\mathbb{P}(H^0(S_1,H)^{\vee})$ to $\mathbb{P}(H^0(S_1,H)^{\vee}_{+1})$.
Recall that $H$ is an ample divisor on $S_1$ (indeed, it is very ample), so the image of $\varphi_{|H|}\times \varphi_{|4F+2O|}$ is the product of $S_1$ and the cone over the rational normal curve of degree $4$. Observe that generically this map is $2:1$, and so it descends to a $1:1$ map on $X$ and on $Y$ . So $\varphi_{|(H + 4F + 2O)_Y|}$ maps $Y$ on the product of $dP$ with the cone over the rational normal curve of degree $4$.

\section{Explicit equations of \texorpdfstring{$Y$}{Y}}\label{sec: explicit}

The aim of this section is to give some explicit equations for the projective models described above, in terms of the corresponding equations for $S_i$.

With a slight abuse, in this section we will substitute  $\mathbb{F}_4$ to its singular model as the cone on the rational normal curve of degree $4$. In this way we will obtain better models for $Y$.

\subsection{}
Let $S_1$ be the double cover of $\mathbb{P}^2_{(x_0:x_1:x_2)}$ whose equation is \begin{equation}\label{eq S_1}w^2=f_6(x_0:x_1:x_2)\end{equation}
so that the curve $C$ is $V(f_6(x_0:x_1:x_2))$.
We assume that $C$ is irreducible, even if some of the following results can be easily generalized. The cover involution $\iota_1$ acts as $(w;(x_0:x_1:x_2))\mapsto (-w;(x_0:x_1:x_2))$.

\subsection{}\label{sect: condition for cy}
Before giving the description of $S_2$, we make a little digression on the Weierstrass equation of an elliptic fibration. In particular, let $Y \longrightarrow V$ be an elliptic fibration and
\begin{equation}\label{eq: weir}y^2 = x^3 + A x + B\end{equation}
an equation for its Weierstrass model. The condition that $Y$ is a Calabi--Yau variety is equivalent to
\[A \in H^0(V, -4K_V), \qquad B \in H^0(V, -6K_V).\]
The discriminant $\Delta$ is then an element of $H^0(V, -12K_V)$.

In particular if $V$ is $\mathbb{P}^m$ (resp. $\mathbb{P}^n\times \mathbb{P}^m$), the functions $A$, $B$ and $\Delta$ are homogeneous polynomials of degree $4m+4$, $6m+6$ and $12m+12$ (resp. of bidegree $(4n+4,4m+4)$, $(6n+6,6m+6)$ and $(12n+12,12m+12)$).

We observe that, if $V$ is $\mathbb{P}^m$ (resp. $\mathbb{P}^n\times \mathbb{P}^m$) requiring that all the singular fibers of the elliptic fibration \eqref{eq: weir} are of type $I_5$ implies that $m\equiv 4\mod 5$ (resp. $n\equiv 4\mod 5$ and $m\equiv 4\mod 5$). 
In case $V$ is a 3-fold, this gives a stronger version of Remark \ref{rem: no only I5}.

\subsection{}
Let $S_2$ be the elliptic K3 surface whose Weierstrass equation is
\begin{equation}\label{eq S_2 weierstrass}
y^2=x^3+A(t:s)x+B(t:s),
\end{equation}
where (according to the previous section) $A(t: s)$, $B(t:s)$ are homogeneous polynomials of degree $8$ and $12$ respectively.
For generic choices of $A(t:s)$ and $B(t:s)$ the elliptic fibration \eqref{eq S_2  weierstrass} has 24 nodal curves as unique singular fibers. For specific choices one can obtain other singular and reducible fibers.
The cover involution $\iota_2$ acts as $(y,x;(t:s))\mapsto (-y,x;(t:s))$.

Equivalently $S_2$ is the double cover of the Hirzebruch surface $\mathbb{F}_4$ given by
\begin{equation}\label{eq s_2 hirz}
u^2=z(x^3+A(t:s)xz^2+B(t:s)z^3)
\end{equation}
where the coordinates $(t,s,x,z)$ are the homogeneous toric coordinates of $\mathbb{F}_4$, see e.g. \cite[$\S$2.3]{CG13}. The action of $\iota_2$ on these coordinates is $(u,t,s,x,z)\mapsto(-u,t,s,x,z)$. Observe that the curve on $\mathbb{F}_4$ defined by $z (x^3 + A(t: s) x z^2 + B(t: s) z^3) = 0$ is linearly equivalent to $-2K_{\mathbb{F}_4}$.

\subsubsection{}\label{subsec: elliptic fibration with I5's}
The choice of particular polynomials in \eqref{eq S_2 weierstrass} is associated to the choice of particular fibers of the fibration.
Indeed, this elliptic fibration has a $I_5$-fiber in $(\overline{t}:\overline{s})$ if and only if the following three conditions hold:
\begin{enumerate}

\item $A(\overline{t}:\overline{s})\neq 0$;
\item $B(\overline{t}:\overline{s})\neq 0$;
\item $\Delta$ vanishes of order 5 in $(\overline{t}:\overline{s})$, where $\Delta:=4A^3+27B^2$.
\end{enumerate}

Up to standard transformations one can assume that the fiber of type $I_5$ is over $t=0$ and
\[ A(t:s):=t^8+\sum_{i=1}^7a_it^is^{8-i}-3s^8,
\]
\begin{align*}
B(t:s):=b_{12}t^{12}+\sum_{i=5}^{11}b_it^is^{12-i}+(-a_4+\frac{a_1^4}{1728}+\frac{a_3a_1}{6}+\frac{a_2^2}{12}+\frac{a_2a_1^2}{72})t^4s^8+ \\  + (-a_3+\frac{a_2a_1}{6}+\frac{a_1^3}{216})t^3s^9+(-a_2+\frac{a_1^2}{12})t^2s^{10}-a_1t^1s^{11}+2s^{12}
.
\end{align*}
We observe that the polynomials $A(t:s)$ and $B(t:s)$ depend on $14$ parameters and, indeed, $14$ is exactly the dimension of the family of K3 surfaces whose generic member has an elliptic fibration with one fiber of type $I_5$.

We already noticed that an elliptic fibration on a K3 surface has at most 4 fibers of type $I_5$ and indeed there are two distinct families of K3 surfaces with this property: the Mordell--Weil group of the generic member of one of these surfaces is trivial, the one of the other is $\Z/5\Z$, \cite[Case 2345, Table 1]{Shim}.

The K3 surfaces of the latter family are known to be double cover of the extremal rational surface $[1,1,5,5]$ whose Mordell--Weil group is $\Z/5\Z$, see \cite[Section 9.1]{SS} for the definition of the rational surface. By this property it is easy to find the Weierstrass equation of the K3 surface (as described in \cite[Section 4.2.2]{BDGMSV17}). Indeed, the equation of the rigid rational fibration over $\mathbb{P}^1_{(\mu)}$ is \begin{equation}\label{eq: Weierstrass R5511} y^2=x^3+A(\mu)x+B(\mu), \ \ \mbox{where}\end{equation}
 $$A(\mu):=-\frac{1}{48}\mu^4-\frac{1}{4}\mu^3\lambda-\frac{7}{24}\mu^2\lambda^2+\frac{1}{4}\mu\lambda^3-\frac{1}{48}\lambda^4,\mbox{  and }$$
$$B(\mu):=\frac{1}{864}\mu^6+\frac{1}{48}\mu^5\lambda+\frac{25}{288}\mu^4\lambda^2+\frac{25}{288}\mu^2\lambda^4-\frac{1}{48}\mu\lambda^5+\frac{1}{864}\lambda^6. $$
In order to obtain the two dimensional family of K3 surfaces we are looking for, it suffices to apply a base change of order two $f:\mathbb{P}^1_{(t:s)}\ra\mathbb{P^1}_{(\mu:\lambda)}$ to the rational elliptic surface. In particular if $f$ branches over $(p_1:1)$ and $(p_2:1)$ the base change $\mu=p_1t^2+s^2$, $\lambda=t^2+s^2/p_2$ produces the required K3 surface if the fibers over $(p_1:1)$ and $(p_2:1)$ of the rational elliptic surface are smooth.

\subsection{The elliptic fibration \texorpdfstring{$\mathcal{E}$}{E}}
Let us now consider the equation \eqref{eq  S_1} for $S_1$ and the equation \eqref{eq S_2  weierstrass} for $S_2$. The action of $\iota_1\times \iota_2$ on $S_1\times S_2$ leaves invariant the functions $Y:=yw^3, X:=xw^2,\ \ x_0,\ \ x_1,\ \ x_2,\ \ t,\ \ s$. Hence an equation for a birational model of $Y$ expressed in these coordinates is
\begin{equation}\label{eq: Y weierstrass cF}
Y^2=X^3+A(t:s)f_6^2(x_0:x_1:x_2)X+B(t:s)f_6^3(x_0:x_1:x_2).
\end{equation}

The previous equation is a Weierstrass form for the elliptic fibration
\[ \mathcal{E}:Y\ra\mathbb{P}^2_{(x_0:x_1:x_2)}\times\mathbb{P}^1_{(t:s)}.
\]
Observe that the coefficient $A(t:s)f_6^2(x_0:x_1:x_2)$ and $B(t:s)f_6^3(x_0:x_1:x_2)$ are bihomogeneous on $\PP^2 \times \PP^1$ of bidegree $(12, 8)$ and $(18, 12)$ respectively, so by \ref{sect: condition for cy} we have another proof that the total space of the elliptic fibration $\mathcal{E}$ is indeed a Calabi--Yau variety.

One can check the properties of this fibration described in Section \ref{subsec: singular fibers} directly by the computation of the discriminant of the Weierstrass equation \eqref{eq: Y weierstrass  cF}, indeed
\[ \Delta(\mathcal{E})=f_6^6(x_0:x_1:x_2)(4A^3(t:s)+27B^2(t:s))=f_6^6(x_0:x_1:x_2)\Delta(\pi).
\]

We observe that in this birational model the basis of the fibration is $\mathbb{P}^2\times\mathbb{P}^1$ and the del Pezzo surface contained in the discriminant is the blow up of $\mathbb{P}^2$ in the singular points of $f_6(x_0:x_1:x_2)$. The singular fibers due to the factor $\Delta(\pi)$ in $\Delta(\mathcal{E})$ are not generically modified by the blow up of $\mathbb{P}^2$ in $n$ points, so that over the generic point of $\mathbb{P}^2$ (and thus of del Pezzo surface) the singular fibers of $\mathcal{E}$ corresponds to singular fibers of $\pi$.

In some special cases it is also possible to write more explicitly a Weierstrass form of this elliptic fibration with basis the product of the del Pezzo surface and $\mathbb{P}^1_{(t:s)}$, as we see in \ref{subsec: n=6} and \ref{subsec: n=5}.

\begin{rem}{\rm  A generalization of this construction produces 4-folds with Kodaira dimension equal to $-\infty$ (resp. $>0$) with an elliptic fibration. Indeed it suffices to consider $S_2$ which is no longer a K3 surface, but a surface with Kodaira dimension $-\infty$ (resp. $>0$) admitting an elliptic fibration with basis $\mathbb{P}^1$. So the equation of $S_2$ is $y^2=x^3+A(t:s)x+B(t:s)$ with $deg(A(t:s))=4m$ and $deg(B(t:s))=6m$ for $m=1$ (resp. $m>2$). The surface $S_2$ admits the elliptic involution $\iota_2$ and $(S_1\times S_2)/\iota_1\times \iota_2$ admits as Weierstrass equation analogous to \eqref{eq: Y weierstrass cF}.}\end{rem}

\subsubsection{$n=6$}\label{subsec: n=6}
Let us assume that $C$ has $n=6$ nodes in general position. In this case the del Pezzo surface $dP$ has degree 3 and is canonically embedded as a cubic in $\mathbb{P}^3_{(y_0:y_1:y_2:y_3)}$. So it admits an equation of the form $g_3(y_0:y_1:y_2:y_3)=0$. The image of $C$ under this embedding is the complete intersection of $g_3=0$ and a quadric $g_2(y_0:y_1:y_2:y_3)=0$ in $\mathbb{P}^3$.

The K3 surface $S_1$ is embedded by $\varphi_{|H|}$ in $\mathbb{P}^4{(y_0:y_1:y_2:y_3:y_4)}$ as complete intersection of a cubic and a quadric, and since it is the double cover of $dP$, its equation is
\begin{eqnarray}\label{eq: S1 if n=6}\left\{\begin{array}{rrrr}
y_4^2&=&g_2(y_0:y_1:y_2:y_3)\\
0&=&g_3(y_0:y_1:y_2:y_3).\end{array}\right.\end{eqnarray}

The involution $\iota_1$ acts on $\mathbb{P}^4$ changing only the sign of $y_4$.

With the same argument as before, this leads to the following equation for a birational model of $Y$:
\begin{eqnarray}\label{eq: Y if n=6}\left\{\begin{array}{l}
Y^2=X^3+A(t:s)g_2^2(y_0:y_1:y_2:y_3)X+B(t:s)g_2^3(y_0:y_1:y_2:y_3)\\
g_3(y_0:y_1:y_2:y_3)=0.\end{array}\right.\end{eqnarray}
The first equation is the Weierstrass form of an elliptic fibration with basis $\mathbb{P}^3\times\mathbb{P}^1$ and the second equation corresponds to restrict this equation to the del Pezzo surface embedded in the first factor (i.e. in $\mathbb{P}^3$).

\begin{cor}\label{corollary explcit equation}
The equation
\begin{align*}
\left\{\begin{array}{l}
Y^2=X^3+\left(\sum_{i=0}^8a_it^is^{8-i}\right)g_2^2(y_0:y_1:y_2:y_3)X+\left(\sum_{i=0}^{12}b_it^is^{12-i}\right)g_2^3(y_0:y_1:y_2:y_3)\\
g_3(y_0:y_1:y_2:y_3)=0.\end{array}\right.
\end{align*}
where $g_i$ is an homogenous polynomial of degree $i$ in $\C[y_0:y_1:y_2:y_3]$,
\[ a_0=-3, \, b_0=2, \, b_1=-a_1, \, b_2=-a_2+\frac{a_1^2}{12}, \, b_3=-a_3+\frac{a_2a_1}{6}+\frac{a_1^3}{216}
\]
\[ \mbox{and } b_4=-a_4+\frac{a_1^4}{1728}+\frac{a_3a_1}{6}+\frac{a_2^2}{12}+\frac{a_2a_1^2}{72},
\]
describes a birational model of a Calabi--Yau 4-fold with an elliptic fibration such that the fibers over the del Pezzo surface $(g_3(y_0:y_1:y_2:y_3)=0)\times (t=0)\subset\mathbb{P}^3\times\mathbb{P}^1_{t}$ are generically of type $I_5$.
\end{cor}

The other singular fibers are described by the zeros of the discriminant $$g_2^6(y_0:y_1:y_2:y_3)\left(4\left(\sum_{i=0}^8a_it^is^{8-i}\right)^3+27\left(\sum_{i=0}^{12}b_it^is^{12-i}\right)^2\right).$$
\begin{rem}{\rm With the same process one obtains the equation of elliptic fibration over $dP\times \mathbb{P}^1$ such that there are $m\leq 4$ del Pezzo surfaces in $dP\times \mathbb{P}^1$ over each of them the general fiber is of type $I_5$. To do this it suffices to specialize the coefficients $a_i$, $b_i$ according to the conditions described in Section \ref{subsec: elliptic fibration with I5's}.
In case $m=4$ there are two different specializations, one of them is associated to the presence of a 5-torsion section and its equation is the given in Section \ref{subsec: elliptic fibration with I5's}.}\end{rem}

\subsubsection{$n=5$}\label{subsec: n=5}
Similarly we treat the case $n=5$. So
let us assume that $C$ has $n=5$ nodes in general position. In this case the del Pezzo surface $dP$ has degree 4 and is canonically embedded in $\mathbb{P}^4_{(y_0:y_1:y_2:y_3:y_4)}$ as complete intersection of two quadrics $q_2=0$ and $q_2'=0$. The image of $C$ under this embedding is the complete intersection of the del Pezzo with a quadric $q''_2=0$.

The K3 surface $S_1$ is embedded by $\varphi_{|H|}$ in $\mathbb{P}^5{(y_0:y_1:y_2:y_3:y_4:y_5)}$ as complete intersection of three quadrics, and since it is the double cover of $dP$, its equation is
\begin{eqnarray}\label{eq: S1 if n=5}
\left\{\begin{array}{rrrr}
y_5^2&=&q''_2(y_0:y_1:y_2:y_3:y_4)\\
0&=&q'_2(y_0:y_1:y_2:y_3:y_4)\\
0&=&q_2(y_0:y_1:y_2:y_3:y_4).
\end{array}\right.
\end{eqnarray}

The involution $\iota_1$ acts on $\mathbb{P}^5$ changing only the sign of $y_5$.

Hence a birational model of $Y$ is:
\begin{align}\label{eq: Y if n=5}
\left\{\begin{array}{l}
Y^2=X^3+A(t:s){q''_2}^2(y_0:y_1:y_2:y_3:y_4)X+B(t:s){q''_2}^3(y_0:y_1:y_2:y_3:y_4)\\
q_2'(y_0:y_1:y_2:y_3:y_4)=0\\
q_2(y_0:y_1:y_2:y_3:y_4)=0.
\end{array}\right.
\end{align}
The first equation is the Weierstrass form of an elliptic fibration with basis $\mathbb{P}^4\times\mathbb{P}^1$ and other two equations correspond to restrict this equation to the del Pezzo surface embedded in the first factor (i.e. in $\mathbb{P}^4$).

\begin{rem}{\rm It is possible to obtain explicit equations for the elliptic fibrations with fiber(s) of type $I_5$ as in Corollary \ref{corollary explcit equation}.}\end{rem}

\subsection{The double cover \texorpdfstring{$Y\ra \mathbb{P}^2\times\mathbb{F}_4$}{Y -> P2 x F4}}

Let us consider the equation \eqref{eq  S_1} for $S_1$ and \eqref{eq s_2  hirz} for $S_2$. The following functions are invariant for $\iota_1\times \iota_2$
$$W:=uw,\ \ x_0,\ \ x_1,\ \ x_2,\ \ t,\ \ s,\ \ x,\ \ z$$
and they satisfy the equation
\begin{equation}\label{eq double cover Y}
W^2=f_6(x_0:x_1:x_2)z(x^3+A(t:s)xz^2+B(t:s)z^3).
\end{equation}
This equation exhibits a biration model of $Y$ as double cover of the rational 4-fold $\mathbb{P}^2\times\mathbb{F}_4$ branched over a divisor in $|-2K_{\mathbb{P}^2\times\mathbb{F}_4}|$. In particular this is the equation associated to the linear system $|(h+4F+2O)_Y|$.

The projections of \eqref{eq double cover Y} gives different descriptions of projective models: the one associated to the linear system $|\delta_h|$ is obtained by the projection to $\mathbb{P}^2$;
the one associated to $|\delta_{4F+2O}|$ is obtained by the projection to $\mathbb{F}_4\subset \mathbb{P}^5$;
the one associated to the linear system $|\delta_F|$ is obtained to the projection to $\mathbb{P}^1_{(t:s)}$.

Consider first the composition with the projection on $\PP^2$ to obtain an equation for $\mathcal{G}$. Fix a point $(\bar{x}_0: \bar{x}_1: \bar{x}_2) \in \PP^2$ and assume that $f_6(\bar{x}_0: \bar{x}_1: \bar{x}_2) \neq 0$. Then the corresponding fiber has equation
\[W^2 = f_6(\bar{x}_0: \bar{x}_1: \bar{x}_2) z (x^3 + A(t: s) x z^2 + B(t: s) z^3),\]
which is easily seen to be isomorphic to $S_2$ (substitute $W$ with $\sqrt{f_6(\bar{x}_0: \bar{x}_1: \bar{x}_2)}W$ to find an equation equivalent to \eqref{eq s_2 hirz}).

Consider now the composition with the projection on $\mathbb{F}_4$. Fix a point $(\bar{t}, \bar{s}, \bar{x}, \bar{z}) \in \mathbb{F}_4$ which does not lie on the negative curve nor on the trisection. Then the corresponding fiber is
\[W^2 = f_6(x_0: x_1: x_2) \bar{z} (\bar{x}^3 + A(\bar{t}: \bar{s}) \bar{x} \bar{z}^2 + B(\bar{t}: \bar{s}) \bar{z}^3),\]
which is a $K3$ surface isomorphic to $S_1$.

Finally we give an equation for $\mathcal{H}$. Let us put $z=1$ in \eqref{eq double cover Y} and perform the change of coordinates $w\mapsto w/f_6$, $x\mapsto x/f_6$. Multiplying the resulting equation by $f_6^2$, we obtain $$w^2=x^3+A(t:s)f_6^2(x_0:x_1:x_2)x+B(t:s)f_6^3(x_0:x_1:x_2).$$
For every fixed $(\overline{t}:\overline{s})\in\mathbb{P}^1$, this is the equation of a Calabi--Yau 3-fold of Borcea--Voisin type obtained from the K3 surface $w^2=f_6(x_0:x_1:x_2)$ and the elliptic curve $y^2=x^3+A(\overline{t}:\overline{s})x+B(\overline{t}:\overline{s})$, see \cite[Section 4.4]{CG13}.

\subsubsection{}
We now want to describe what happens if the sextic curve in $\PP^2$ has $n = 6$ or $n = 5$ nodes.

Assume first that $\rho': S_1 \longrightarrow \PP^2$ is branched along a sextic with $6$ nodes. Then we can use \eqref{eq: S1 if n=6} and \eqref{eq s_2 hirz} to describe $S_1$ and $S_2$ respectively, and using the same argument as before (i.e.\ put $W = y_4 u$) we obtain the equation
\[\left\{ \begin{array}{l}
W^2 = g_2(y_0: y_1: y_2: y_3) z (x^3 + A(t:s) x z^2 + B(t: s) z^3)\\
0 = g_3(y_0: y_1: y_2: y_3)
\end{array} \right.\]
which exhibits $Y$ as double cover of $dP\times \mathbb{F}^4$. 
Let us denote by 
$U\ra \mathbb{P}^3\times \mathbb{F}_4$ the double cover branched on $g_2(y_0: y_1: y_2: y_3) z (x^3 + A(t:s) x z^2 + B(t: s) z^3)$. The branch divisor is $2H_{\mathbb{P}^3}-2K_{\mathbb{F}_{4}}$ and so $Y$ is a section of the anticanonical bundle of $U$. 

With a further change of variables, where the only non-identic transformation are $W' = g_2 W$ and $x' = g_2 x$, we then find the following equation for a birational model of $Y$ (we drop the primes for simplicity of notation)
\[\left\{ \begin{array}{l}
W^2 = z (x^3 + A(t:s) g_2^2(y_0: y_1: y_2: y_3) x z^2 + B(t: s) g_2^3(y_0: y_1: y_2: y_3) z^3)\\
0 = g_3(y_0: y_1: y_2: y_3).
\end{array} \right.\]
Here the first equation gives an elliptic fibration over $\PP^3 \times \PP^1$ as a double cover, while the second restricts this fibration to $dP \times \PP^1$.

Analogously, if $n = 5$, then $S_1$ and $S_2$ are described by \eqref{eq: S1 if n=5} and \eqref{eq s_2 hirz} respectively, so that we have the following equation for $Y$:
\[\left\{ \begin{array}{l}
W^2 = q_2'' z (x^3 + A x z^2 + B z^3)\\
0 = q_2'\\
0 = q_2,
\end{array} \right.\]
with the same considerations as the case just treated.

\subsection{An involution on $Y$} By construction $Y$ admits an involution $\iota$ induced by $\iota_1\times \id\in \Aut(S_1\times S_2)$ and acting as $-1$ on $H^{4,0}(Y)$. Since $\iota_1\times \id=(\iota_1\times \iota_2)\circ(\id\times \iota_2)$, $\iota$ is equivalently induced by $\id\times \iota_2$. 
The involution $\iota$ has a clear geometric interpretation in several models described above. By \ref{subsec: projective models}, $Y$ is a $2:1$ cover of $\mathbb{P}^2\times \mathbb{F}_4$ whose equation is given in \eqref{eq double cover Y}. The involution $\iota$ is the cover involution, indeed it acts as $-1$ on the variable $W:=uw$ and by \eqref{eq  S_1} $\iota_1\times \id$ acts as $-1$ on $w$.

By \ref{subsec:  fibrations}, $Y$ admits the elliptic fibration $\mathcal{E}$ whose equation is given in \eqref{eq: Y weierstrass cF}. The involution $\iota$ is the cover involution, indeed it acts as $-1$ on the variable $Y:=yw^3$ and by \eqref{eq S_2 weierstrass} $\id\times \iota_2$ acts as $-1$ on $y$. 

Hence $Y/\iota$ is birational to $\mathbb{P}^2\times \mathbb{F}_4$ and admits a fibration in rational curves, whose fibers are the quotient of the fibers of the elliptic fibration $\mathcal{E}$.


\bigskip

\bigskip

\medskip

\bigskip

\noindent \textbf{Andrea Cattaneo}\\

\noindent    Dipartimento di Scienza e Alta Tecnologia, Universit\`a degli Studi dell'Insubria, Via Valleggio, 11, I-22100 Como, Italy. \\
\verb|andrea1.cattaneo@uninsubria.it| \\

\noindent \textbf{Alice Garbagnati}\\

\noindent Dipartimento di Matematica \emph{``Federigo Enriques''}, Universit\`{a} degli Studi di Milano, Via Saldini 50, I-20133 Milano, Italy. \\
\verb|alice.garbagnatii@unimi.it| \\

\noindent \textbf{Matteo Penegini}\\

\noindent Dipartimento di Matematica DIMA, Universit\`a degli Studi di Genova, Via Dodecaneso 35, I-16146 Genova, Italy. \\
 \verb|penegini@dima.unige.it|  \\

\end{document}